\documentclass{article}
\usepackage{amsthm, a4wide}
\usepackage{amsfonts}
\usepackage{amsmath}
\usepackage{amssymb}
\newtheorem{theorem}{Theorem}[section]
\newtheorem{corollary}[theorem]{Corollary}

\newtheorem{example}{Example}
\newtheorem{lemma}[theorem]{Lemma}
\newtheorem{proposition}[theorem]{Proposition}
\newtheorem{remark}[theorem]{Remark}
\newtheorem{step}{Step}
\newtheorem{assumption}{Assumption}
\newcommand{\eps}{\varepsilon}
\newcommand{\R}{I\!\!R}
\newcommand{\E}{\mathbf{E}}
\newcommand{\Q}{\mathbb{Q}}
\newcommand{\Lu}{\mathcal{L}}
\newcommand{\W}{\mathcal{W}}
\newcommand{\tr}{\text{trace}}\newcommand{\supp}{\text{supp}}

\newcommand{\Mi}{{\mathbf M}}
\begin{document}
\title{Convergence in Multiscale Financial Models with Non-Gaussian Stochastic Volatility.\footnotemark[2]
 }
\author{Martino Bardi\footnotemark[1], Annalisa Cesaroni\footnotemark[1], 
Andrea Scotti\footnotemark[1] }
  \date{}
 
\maketitle
\renewcommand{\thefootnote}{\fnsymbol{footnote}}
\footnotetext[1]{Dipartimento di Matematica, Universit\`a
di Padova, via Trieste 63, 35121 Padova, Italy ({\tt
bardi@math.unipd.it, 
 acesar@math.unipd.it, andrea.scotti@studenti.unipd.it}).}
\footnotetext[2]{Partially supported by 
the Fondazione CaRiPaRo Project "Nonlinear Partial Differential Equations: models, analysis, and control-theoretic problems"  and the European Project Marie Curie ITN "SADCO - Sensitivity Analysis for Deterministic Controller Design".}
\renewcommand{\thefootnote}{\arabic{footnote}}
 
\begin{abstract}
We consider stochastic control systems affected by a fast mean reverting volatility $Y(t)$ driven by a pure jump L\'evy process.
Motivated by a large literature on financial models, we assume that $Y(t)$ evolves at a faster time scale $t/\epsilon$ than the assets, and we study the asymptotics as $\epsilon\to 0$.
This is a singular perturbation problem that we study mostly by PDE methods within the theory of viscosity solutions.

\smallskip


\end{abstract}
\maketitle

\section{Introduction}
 We consider the  controlled stochastic differential system in $\R^{n+1}$
\begin{equation}\label{sdeintro} \begin{cases} 
dX(s)=f(X(s), Y_\eps(s^-), u(s)) ds+\sigma(X(s),Y_\eps(s^-), u(s))dW(s)   \\ 
 dY_\eps(s)=   -\frac{1}{\varepsilon }Y_\eps(s^-)ds+dZ\left(\frac{1}{\varepsilon }s\right)      \\  \end{cases}
\end{equation}   with  $s\geq t$ and initial data $X(t)=x\in\R^n$, $Y_\eps(t)=y\in\R$. The function $u(\cdot)$ is the control taking values in a given compact set $U$,  $\eps>0$ is a small parameter, $W$ is a standard $r$-dimensional Brownian motion and $Z$ is a $1$-dimensional  pure jump L\'evi process, independent of $W$. 
We associate to this 
 system  a payoff functional of the form
\[
\E[e^{c(t-T)} g(X(T)
)\ |\ X(t)=x, Y_\eps(t)=y] , \qquad 0\leq t\leq T, 
\] 
where $c\geq 0$ and $g:\R^{n}\to \R$ is a  continuous function with quadratic growth, that we wish to maximise among admissible control functions $u(\cdot)$. 
The value function of this optimal control problem  is defined as 
\begin{equation}\label{valueintro} 
V^\eps(t,x,y):=\sup_{u(\cdot)}\E[e^{c(t-T)} g(X(T)
)].
\end{equation} 
We are interested in the analysis of the limit as $\eps \to 0$ of the   control problem given by the system \eqref{sdeintro} and  by the value function \eqref{valueintro}. 

Our main motivation comes from financial models with stochastic volatility. In such models 
 $X(s)$ represents, for instance, 
 the log-prices of $n$
  assets, or the wealth of a portfolio.
  The
 volatilities of the assets, collected in the matrix $\sigma$, 
 are affected by another  process 
 $Y(s)$ that is usually a diffusion driven by another Brownian motion 
  negatively correlated  with the one driving the stock prices.
  Fouque, Papanicolaou, and  Sircar 
 argued in  the book \cite{FPS} that the bursty  behaviour of volatility observed in financial markets can be described by introducing a faster time scale for a mean-reverting diffusion process $Y$. 
 Several result along these lines were found, mostly for problems without controls, e.g., in \cite{fps1, fps2, 
 fpss}, 
  see also  the references therein. The papers \cite{BCM, bc} by the first two authors introduced viscosity methods to prove the convergence of the singular perturbation for models involving a control variable and therefore associated to a fully nonlinear Hamilton-Jacobi-Bellman 
   equation. The main example in \cite{BCM} was the classical Merton's portfolio optimization problem with volatility depending on an ergodic diffusion process $Y$.
  
On the other hand, the work by  Barndorff-Nielsen and  Shephard \cite{BNS1} showed that processes of Ornstein-Uhlenbeck type driven by a pure-jump L\'evy process are more appropriate models of the volatility than diffusions. 
Several authors studied financial problems with such a non-Gaussian mean-reverting stochastic volatility evolving at the same time scale as the prices: Merton's problem in 
 \cite{BKR} and various option pricing problems in \cite{NV, Hs2, hs}. The novelty of the present paper is combining multiple scales and stochastic volatility with jumps.  In particular, we extend to this context the 
 convergence results for asset pricing and for Merton's problem obtained in  \cite{BCM}. 

To conclude this brief bibliographical introduction we refer 
 to the books \cite{SATO} for the theory of L\'evy processes and \cite{CT} for their applications to finance.   Let us mention also that the recent paper \cite{LL} deals with a multiscale model whose assumptions are in a sense opposite to ours: the volatility is a diffusion and the slow variable $X$ is driven by a jump process.

We describe now in more details our result. 
By dynamic programming 
 arguments \cite{Ph2, S},  the value function $V^\eps$ in \eqref{valueintro} is a viscosity solution of the integro-differential Hamilton-Jacobi-Bellman equation
\begin{eqnarray}\label{pdeintro}   - V^\eps_t &+& H(x,y, D_x V^\eps, D^2_{xx}V^\eps)    
-\frac{1}{\eps} y\cdot D_y V^\eps  \\&-&\frac{1}{\eps}      \int_0^{+\infty} (V^\eps(t,x,z+y)-V^\eps(t,x,y)- D_yV^\eps (t,x,y)\cdot z \,1_{ |z| \leq 1}) d\nu (z)+cV^\eps= 0 \nonumber \end{eqnarray} 
in $(0,T)\times\R^n\times\R$, with terminal data $V^\eps(T,x,y)=g(x)$, 
where $\nu$ is the L\'evy measure associated to the process $Z$ and $H$ is a standard Hamiltonian associated to stochastic control problems, see the precise definition  \eqref{H} in Section \ref{4}.
Letting $\eps\to 0$ in \eqref{pdeintro} is a singular perturbation problem, and we treat it by methods of the theory of viscosity solutions to integro-differential PDEs. 
 Our main result, see Theorem \ref{ConvTheo},  is the proof of the uniform convergence of $V^\eps$ as $\eps\to 0$ to  the unique viscosity solution $V(t,x)$ of the {\em effective PDE} 
\begin{equation} \label{effintro}
 -V_t
+ \int_{\R} H(x, y, D_x V, D^2_{xx} V) d\mu (y) + cV
=0
\end{equation}
in $(0,T)\times\R^n$, with terminal data $V(T,x)=g(x)$, where $\mu$ is the unique invariant measure of the process $Y_\eps$ (which  is independent of $\eps$, see  Proposition \ref{lem:2}), and $H$ is the Hamiltonian appearing in \eqref{pdeintro}. 
 
The second step in the solution of the singular perturbation problem is  to interpret  the effective equation \eqref{effintro} as the Hamilton-Jacobi-Bellman equation for a  limit {\em effective control problem}.
This  can be done (Remark \ref{effcontrol}) by a general relaxation procedure taken from \cite{BTer}. However a simpler and explicit 
representation of the effective systems can be given for our two main model problems, that is, the asset pricing and the Merton's optimisation problem, see Section \ref{7}. In particular we show that in the limit $\eps\to 0$ the asset pricing problem converges to a new asset pricing with constant volatility
    $$\tilde{\sigma}^2:=
    \int_{\R} \sigma^2(y) \mu(dy) ,
    $$
    where $\mu$ is the invariant measure of the process $Y_\eps$, whereas the portfolio optimisation converges to a Merton's problem with 
    constant volatility given by the harmonic mean of $\sigma(y)$, namely,
    $$
    \overline{\sigma}^2:=
    \left(\int_{\R}\frac{1}{\sigma^2(y)}\mu(dy)
    \right)^{-1
    } ,
  $$
  which is smaller than $\tilde{\sigma}$.

The proofs of our results  rely on several tools, 
  among which we mention the exponential ergodicity of the fast process $Y$, proved  by   Kulik \cite{K}, some properties of viscosity solutions for integrodifferential equations \cite{
  Sa, Ph2, BI, C}, and the perturbed test function method introduced by Evans for periodic homogenization \cite{E} and extended to singular perturbations in \cite{ab, AB2}. To adapt this method to the current setting of unbounded fast variables $y$, we suitably modify  the  perturbed test function by means of a Lyapunov function associated to the process $Y$ (this is an improvement also with respect to \cite{BCM, bc}). 
  
  Our approach is flexible enough to deal with more general problems, such as integral payoffs and slow variables $X$ depending also on a jump process (under more restrictive growth conditions on the data). Moreover the  fast variables $Y$ could be assumed to be vectorial and depending on 
  a combination of jump and diffusion processes, provided that the resulting process be uniformly ergodic and its generator satisfy the strong maximum principle. The last requirement is a limitation to the applicability of our method: if tested on processes $Z$ whose L\'evy measure $\nu$ has $\supp(\nu)\subseteq [0, +\infty)$ it forces us to assume 
   \[
   \int_{|z|\leq 1} |z|\nu (dz)=+\infty ,
    \]
    whereas the opposite case is treated, e.g., in \cite{BKR}.  On the other hand, all processes $Z$ generated by a fractional Laplacian fit our assumptions.

    The paper is organized as follows. Section \ref{2} describes the basic assumptions on the optimal control problem. Section \ref{3} is devoted to the assumptions on the volatility process $Y$ and its properties, in particular the exponential ergodicity and the strong maximum principle and Liouville property of the generator.
    Section \ref{4} describes the partial integrodifferential HJB equation associated to $V^\eps$. In Section \ref{5} we study the cell problem that allows to identify the effective Hamiltonian for the limit PDE. Section \ref{6} contains the statement and proof of the convergence theorem. In Section \ref{7} we apply the previous theory to financial models.  
  

\section{Standing assumptions on the  control system}
\label{2}
 We consider the  controlled stochastic differential equation  
\begin{equation}\label{sde} \begin{cases} 
dX(s)=f(X(s), Y_\eps(s^-), u(s)) ds+\sigma(X(s),Y_\eps(s^-), u(s))dW(s)  & X(t)=x 
 \in \R^n\\
 dY_\eps(s)=   -\frac{1}{\varepsilon }Y_\eps(s^-)ds+dZ\left(\frac{1}{\varepsilon }s\right)  &Y_\eps(t)
 \in \R \\  \end{cases}\end{equation}
where  $\eps>0$ is a small parameter,  $W(t)=(W^1(t),\dots,W^r(t))$ is a r-dimensional Brownian motion and $Z(t)$ is a pure jumps L\`evy process. Furthermore, we assume that $W$ and $Z$ are independent. 

We assume the standard conditions on coefficients  $f:\R^n\times\R\times U\to\R^n$, $\sigma:\R^n\times\R\times U\to \Mi^{n,r}$, where
$\Mi^{n,r}$ denotes the  set of $n\times r$ matrices. In particular we assume that $f,\sigma$ are  continuous functions, Lipschitz continuous in $(x,y)$ uniformly w.r.t. $u\in U$,  where $U$ is a compact set  and   that $f(x,\cdot, u)$, $\sigma(x,\cdot, u)$ are bounded for every $(x,u)$. 
We say that a process $u(\cdot)$ on the time interval $[0,T]$ is an admissible control function if it takes values in $U$ and is progressively measurable with respect to the filtration generated by $W(\cdot)$ and $Z(\cdot)$, and we set 
 \[
 \mathcal{U}:=\{u(\cdot)\ \text{   admissible control function}\}.
 \]

We will not make any non-degeneracy assumption on $\sigma$. 

To fit financial models, we assume that the $i$-th component of the velocity of $X$ vanishes when $X_i(s)=0$, i.e., 
\begin{equation}
\label{zero}
 x_i=0 
 \quad \Longrightarrow \quad f_i(x,y,u)=0, \; \sigma_{ij}(x ,y,u)=0, \quad\forall\, j=1,\dots,r, \, y\in\R,\, u\in U .  
\end{equation}
We set 
\[
\R^n_+:=\{x\in\R^n\,:\,x_i>0 \text{ for every }i=1,\dots,n\}.
\]
Note that under the previous assumption, then the set $\overline{\R^n_+}\times \R$ is invariant under the stochastic process \eqref{sde}. This means that if the initial data satisfies $x_i\geq 0$ for some $i$, then every solution to \eqref{sde} satisfies  $X_i(s)\geq 0$ almost surely for every $s\geq t$. 

We consider for simplicity a payoff functional depending only on the position of the system at a fixed terminal time $T>0$.   The utility function $g:\R^n_+\to\R$ is a continuous  function satisfying
\begin{equation}
\label{GrowthC}
\exists\,K>0,   \quad \text{such that } \quad |g(x)|\le K(1+|x|^2) \quad \text{for every } x\in\R^n_+  \\
\end{equation}   
and the discount factor is $c\geq 0$. 
 
Therefore the value function of the optimal control problem is
\begin{equation}
\label{PO}
V^\eps(t,x,y):=\sup_{u\in\mathcal{U}}\E[e^{c(t-T)}g(X(T))\,|\, X(t)=x,\, Y_\eps(t)=y], \\
\end{equation}
and $(X(\cdot),\,Y_\eps(\cdot))$ satisfy \eqref{sde} with control $u$. This choice of the payoff is sufficiently general for the application to finance models presented in this paper, but we could easily include in the payoff an integral term keeping track of some running costs or earnings.

\section{The fast subsystem}\label{3}
We consider the process $Y$ in \eqref{sde}, putting $\eps=\frac{1}{\lambda}>0$. $Y$ is an Ornstein-Uhlenbeck non-Gaussian process driven by a L\'evy process $Z$, \emph{i.e.}
\begin{equation}
\label{Fast}
d Y_\lambda(s)=-\lambda Y_\lambda(s^-) ds + dZ(\lambda s), \\
\end{equation} 
where $\lambda>0$ is the rate of mean reversion.
We assume that the process $Z$ is a 
pure-jump 
 L\'evy process with no drift, i.e., a L\'evy process whose L\'evy-Ito decomposition has null continuous part, and we choose its cadlag (RCLL) version.
 We refer to the monograph \cite{SATO}  for a general introduction to L\'evy processes.   

The process $Y_\lambda(s)$, with initial datum $Y_\lambda(0)=y\in\R$, can be explicitly written as
\[
Y_\lambda(s)=ye^{-\lambda s}+\int_0^se^{\lambda(u-s)} dZ(\lambda u).  
\]
We associate to the process $Z(s)$ its L\'evy measure  $\nu$. Intuitively speaking, the L\'evy  measure describes the expected number of
jumps of a certain height in a time interval of length 1, i.e. 
\[
\nu(B)=\E(\#\{s\in[0,1],\,\,Z(s)-Z(s^-)\neq 0,\,\,Z(s)-Z(s^-)\in B \} ). \]

The L\'evy  measure has no mass at the origin, while singularities (i.e. infinitely many jumps) can
occur around the origin (i.e. small jumps). Moreover, the mass away from
the origin is bounded (i.e. only a finite number of big jumps can occur).   
In  particular $\nu$ satisfies the following integrability condition (see \cite{SATO}): 
\begin{equation} 
\label{int}
\int_{|z|<1}|z|^2\nu(dz)+\int_{|z|\geq 1}\nu(dz) <\infty. 
\end{equation}
 We will consider L\`evy processes with infinite activity, that is, such that $\nu(\R) = +\infty$. In this case  almost all paths of the process $Z$ have an infinite number of jumps on every compact interval.

The infinitesimal generator of this process $Y_\lambda$ (see \cite{SATO}, Thm. 31.5) is given by $\lambda \mathcal{I}$, where $\mathcal{I}$ is defined as follows 
\begin{equation}\label{l} \Lu(y,  [f]  )  =  -   f'(y)\cdot  y    +   \int_0^{+\infty} (f(z+y)-f(y)-f'(y)z 1_{ z \leq 1}) d\nu (z). \end{equation}

\subsection{Standing assumptions on the L\'evy process} 
Besides the integrability condition \eqref{int} we need to assume some other condition on the measure $\nu$. 
The first one is a quite standard non degeneracy condition (see \cite{BI}). In particular we assume that the measure is singular at $0$ and we introduce a parameter  $p\in (0,2)$ characterizing the singularity of the measure at $0$.   
\begin{assumption} 
\label{A1}
There exist  $C>0$,  $p\in(0,2)$ such that for every $0<\delta\le1$ 
 \begin{equation}\label{a1} 
\int_{|z|\leq \delta} |z|^2 \nu(dz) \ge C \delta^{2-p}.
\end{equation}
\end{assumption}

\begin{remark}
\upshape An equivalent way  to state \eqref{a1}  is the following: there exists $r\in(0,2)$ such that
\begin{equation} \label{a1n} 
\delta^{-r}\int_{|z|\leq \delta} |z|^2 \nu(dz) \to +\infty \quad\quad \text{as } \delta\to 0^+. \\
\end{equation} Indeed  if \eqref{a1} holds then  \eqref{a1n} is satisfied   for   $r\in (0,2-p)$. Viceversa, if \eqref{a1n} holds, then \eqref{a1} is satified with $p\in(0, 2-r)$.  
\end{remark}

\begin{remark}\upshape \label{crit}
There is a general difference between the case in which the constant $p$ in \eqref{a1}
is above or below the critical value  $p=1$.
 In particular if  $p\in (0,1)$, using \eqref{a1n}, it can be proved that \[\int_{|z|\leq 1} |z|\nu (dz)<+\infty, \] and then that almost all   paths of $Z$ have finite variation. If  $p\in (1,2)$, on the other side,    \[\int_{|z|\leq 1} |z|\nu (dz)=+\infty, \] and   almost all   paths of $Z$ have infinite variation (see \cite[Thm. 21.9]{SATO}). 
\end{remark}

We assume a stronger integrability condition at infinity than \eqref{int} of the measure $\nu$. 
\begin{assumption} 
\label{A3}
There exists $q>0$ such that
\begin{equation}\label{a3}
\int_{|z|>1}|z|^q \nu(dz) < +\infty .
 \end{equation}\end{assumption} 
 It can be proved (see \cite[Thm. 25.3]{SATO}) that under Assumption \ref{A3}, the process $Z(t)$   has finite $q$-th moment for every $t$, that is, $\E (|Z_t|^q)<\infty$.

As we will prove in Section \ref{erg} Assumptions \ref{A1},\ref{A3} are sufficient to prove unique ergodicity of the process $Y_\lambda$ defined in \eqref{Fast}. 
Nevertheless,  we need to add  a technical assumption. Indeed the  non degeneracy assumption \eqref{a1} is not sufficient to assure the validity of a Strong Maximum Principle for the associated operator $\Lu$ (see Remark \ref{subordinator}). So we assume also the following. 
\begin{assumption} \label{A2} 
At least one of the following conditions holds:\begin{itemize} 

 \item[i)]   the parameter $p$  in \eqref{a3} satisfies $p\in (1,2)$, 
\item[ii)]  $\R$ can be covered by translations of the support $\supp (\nu)$ of the measure $\nu$
, \emph{i.e.},
\begin{equation}\label{assu}
\R=\bigcup_{n\ge 0}\left(\,\underbrace{\supp(\nu)+\dots+\supp(\nu)}_n\,\right) .
\end{equation} 
\end{itemize} 
\end{assumption}
\begin{remark}
\upshape \eqref{assu}  can be replaced by the requirement that 
$0$ belongs to the topological interior of the measure support $\supp(\nu)$. 
\end{remark}

The main examples of   L\'evy processes satisfying the previous assumptions are the \emph{$\alpha$-stable L\'evy processes}, defined on the whole space or restricted to the half space. We recall that $Z(t)$ is an $\alpha$ stable process  if $Z(t) t^{1/\alpha}= Z(1)$ for all $t$, in the sense that the two processes have the same law.

\begin{example}\upshape
Consider the  $\alpha$-stable L\'evy process with measure 
\begin{equation}
\label{FL}
\nu(dz)=\frac{dz}{|z|^{1+\alpha}}, \\
\end{equation}
where $\alpha\in(0,2)$. The generator of this process is the  fractional Laplacian $(-\Delta)^{\alpha/2}$ (see \cite{BI}). Assumption \ref{A1} holds with $p=\alpha$, and Assumption \ref{A3} holds with $q<\alpha$. Finally the case \emph{ii)} \eqref{assu} in Assumption \ref{A2} 
 holds. 

 \end{example} 
\begin{example} 
\upshape Consider the  $\alpha$-stable L\'evy process  
 restricted to the half space, with measure 
\begin{equation}
\label{FLhs}
\nu(dz)=1_{\{z\ge0\}}(z)\frac{dz}{z^{1+\alpha}}. \\
\end{equation}   In this case the condition \emph{ii)} \eqref{assu} in Assumption \ref{A2} 
 is no more valid, so we assume that  $\alpha\in(1,2)$ to ensure  the condition \emph{i)}. 
 \end{example}

\begin{remark}\label{sub1}\upshape We remark that Assumption \ref{A2} cannot hold if $Z$ is a subordinator. 
Subordinators are 1-dimensional L\'evy processes with non-decreasing sample paths.
 It's easy to deduce that a 1-dimensional L\'evy process is a subordinator if and only if its  associated measure satisfies \begin{equation}\label{sub} \nu((-\infty,0])=0\qquad \int_0^{1}|z|\nu(dz)+\int_1^{+\infty}\nu(dz)<+\infty.\end{equation} This implies in particular that $\nu$ cannot satisfy Assumption \ref{A2} (see Remark \ref{crit}). 
Among $\alpha$-stable processes 
with associated  L\'evy measure $\nu$ defined by \eqref{FLhs}
the subordinators are those with $\alpha\in(0,1)$. A notable example is the inverse Gaussian, corresponding to 
$\alpha=\frac12$. It can be defined as $Z(s):=\inf\{t\ |\ W(t)>s\}$, where  $W$ is a $r$-dimensional Brownian motion.
(Note moreover that if  $Z$ is a $\alpha$-stable subordinator and $W$ is an independent Brownian motion, then $W(Z(t))$ is a stable symmetric L\'evy process of exponent $2\alpha$.)
We will explain in Remark \ref{subordinator} why the proof of convergence presented in this paper does not apply to the subordinators.

\end{remark}
\subsection{Ergodic properties} \label{erg} 
 In this section, we prove that under the previous Assumptions on the L\'evy process $Z$, the process $Y$ defined in \eqref{Fast} is ergodic. In order to prove ergodicity of a process,   two principal features should be checked: recurrence of the process outside some large ball and regularity of the transition probability in some bounded domain.  
The first feature can be provided in a quite standard way via an appropriate version of the
Lyapunov criterium (for the relation between existence of a Lyapunov function for the system and recurrence outside compact sets, we refer to the monograph \cite{has}). We show that Assumption \ref{A3} assures the existence of a Lyapunov function. 

\begin{lemma}
\label{Lyapunov}
Let
\[
\mathcal{Q} = \{ v\in\mathcal{C}^2(\R) | \,\, \exists \,\, \tilde{v}\,\, \text{locally bounded } : \int_1^{+\infty} v(y+z)d\nu(z) \le \tilde{v}(y) \quad \forall\, y \}. \\
\]
Under Assumption \ref{A3} there exists $\varphi\in\mathcal{Q}$ and constant  $a>0$  such that
\begin{equation}
-\Lu[y,\varphi] \geq  a\varphi(y)  \quad \text{and} \quad \varphi \to +\infty \quad \text{as} \quad |y| \to +\infty \\
\end{equation}
\end{lemma}
\begin{proof} For the proof we refer to   \cite[Proposition 4.1]{K}. The idea is to construct $\varphi$ such that $\varphi(y)\leq|y|^q$ for all $y\in\R$ and $\varphi(y)=|y|^q$ for $|y|\geq R$, for $R$ sufficiently large.
\end{proof}

Secondly, Assumption \ref{A1}  provides the regularity of the transition probability of the process $Y$. 

\begin{lemma}
\label{lem:1}
Under Assumption \ref{A1},   the Ornstein-Uhlenbeck process $Y_\lambda(t)$ has a $\mathcal{C}^{\infty}$ density with all bounded derivatives for all $t$ and every initial data $y$. 
\end{lemma}
\begin{proof} 
We refer to \cite{P} and \cite{PZ}.
\end{proof}
We prove now exponential ergodicity of the process $Y$. 
\begin{proposition}
\label{lem:2}
Let Assumptions \ref{A1}, \ref{A3} hold.  Then for any $\lambda>0$, the process $Y_\lambda$ in \eqref{Fast} admits a unique invariant distribution $\mu$, which is independent of $\lambda$. Moreover, it is exponentially ergodic, in the sense that there exists a positive constant $C$ such that for every bounded measurable function $f$ there exists $K>0$ for which 
\begin{equation}\label{es}
\left|\frac{1}{t}\int_0^t\E(f(Y_\lambda(s)))ds - \int_{\R}f(z)\mu(dz)\right| \le  K(1+|y|^q)e^{-Ct} \quad \text{as} \, t\to+\infty, \\  
\end{equation}
where $Y_\lambda$ is the solution to \eqref{Fast} with initial data $Y_\lambda(0)=y$ and $q$ is as in Assumption \ref{A3}.  
\end{proposition} 
\begin{proof} 
In \cite[ Thm. 17.5, Cor. 17.9]{SATO} it is proved that the process $Y_\lambda$  for every $\lambda>0$ has an invariant distribution $\mu_\lambda$. Moreover, $\mu_\lambda$ is the distribution of a L\'evy process with drift
\begin{equation}\notag
\gamma=\frac{1}{\lambda}\int_{|y|>1}\lambda\nu(dy)=\int_{|y|>1}\nu(dy), \\
\end{equation}
and L\'evy measure
\begin{equation}\notag
\mu(B)=\frac{1}{\lambda}\int_{\R}\lambda \int_0^{+\infty}1_{B} (e^{- s}y) ds\nu(dy)= \int_{\R} \int_0^{+\infty}1_{B} (e^{- s}y) ds\nu (dy), \\
\end{equation} 
for every $B\in\mathcal{B}(\R)$. So, the time scaling $dZ(\lambda t)$ assures that the invariant distribution $\mu_\lambda$ of $Y_\lambda$ is independent of $\lambda$. The exponential ergodicity \eqref{es} of $Y$ is proved in \cite[Thm. 1.1 and Prop. 0.1]{K}.
 \end{proof}
 \begin{corollary}\label{Ab}Under the same assumptions and notations of Proposition \ref{lem:2},  for every bounded measurable function $f$ there exists a constant $K>0$ such that   \begin{equation} \label{delta} 
\left|\delta\int_0^{+\infty}\E f(Y_\lambda(t))e^{-\delta t}dt- \int_{\R} f(z)\mu(dz )\right|\le  K(1+|y|^q)  \delta\quad \text{as} \, \delta\to 0^+.  
\end{equation} 
 \end{corollary}
 \begin{proof}
The   property \eqref{delta} can be deduced from \eqref{es} using a result of   Abelian-Tauberian type, 
see  \cite[Thm. 10.2]{Si}. For completeness we give the   proof in our  case. 
First of all note that  it is equivalent to prove, instead of \eqref{delta}, that \begin{equation} \label{delta1} 
\left|\delta\int_1^{+\infty}\E f(Y_\lambda(t))e^{-\delta t}dt- \int_{\R} f(z)\mu(dz ) \right|\le  K(1+|y|^q)  \delta\quad \text{as} \, \delta\to 0^+.  
\end{equation}
We denote $\int_{\R} f(z)\mu(dz )=M$.

Fix $y\in\R$ and define $F_y(t):= \int_1^t \E f(Y_\lambda(s))ds$, for $t\geq 1$. By integration by parts, we get 
\begin{equation}\label{uno}
\delta\int_1^{+\infty}e^{-\delta t}\E(f (Y_\lambda(t)))dt= \delta^2\int_{1}^{+\infty}e^{-\delta t}F_y(t)dt  = \delta \int_{\delta}^{+\infty}e^{-s}F_y\left(\frac{s}{\delta}\right)ds.  
\end{equation} Note that, for $\delta>0$ fixed,  by \eqref{es} \[\frac{\delta}{s} F_y\left(\frac{s}{\delta}\right) \to  M \qquad\text{ as  } s\to +\infty. \] Therefore $F_0=\max_{s\geq\delta} \left|\frac{\delta}{s} F_y\left(\frac{s}{\delta}\right) \right|<\infty$ and  then \[  \delta \ e^{-s}F_y\left(\frac{s}{\delta}\right)= e^{-s}s \frac{\delta}{s}F_y\left(\frac{s}{\delta}\right)\in L^1(\delta, +\infty).\]

 By \eqref{es} we obtain 
\begin{equation}\label{due}
\left|\frac{\delta}{s} F_y\left(\frac{s}{\delta}\right) -M\right|\le K(1+|y|^q)e^{-C\frac{s}{\delta}}+\frac{\delta}{s}\|f\|_\infty, \\
\end{equation}
for $s$ fixed. Therefore, using \eqref{uno} and \eqref{due} we get 
\begin{align*} & \left|\delta\int_1^{+\infty}e^{-\delta t}\E(f (Y_\lambda(t)))dt-M\right| & \\ \leq &  \int_{\delta}^{+\infty}
e^{-s}s\left| \frac{\delta}{s}F_y\left(\frac{s}{\delta}\right)-M\right|ds + (1-e^{-\delta}+\delta e^{-\delta})M    &  \\ \leq  &
\int_{\delta}^{+\infty}
e^{-s}s K(1+|y|^q)e^{-C\frac{s}{\delta}}ds+\int_{\delta}^{+\infty}e^{-s}\delta \|f\|_\infty ds + (1-e^{-\delta}+\delta e^{-\delta})M &   \\  \leq  &
 K(1+|y|^q)e^{-C} \delta^2\frac{C+1}{C^2}+ e^{-\delta}\delta \|f\|_\infty + (1-e^{-\delta}+\delta e^{-\delta})M  
\end{align*}
which gives the desired result. 
\end{proof}

\subsection{Liouville property} 
Finally we state and prove a Liouville property for our integro-differential operator $-\Lu$. This is based on the following Strong Maximum Principle.

 \begin{theorem}[Strong Maximum Principle] \label{SMPFL}
Assume the measure $\nu$ satisfies Assumptions \ref{A1} and \ref{A2}. Let $u\in USC(\R)$ be a viscosity subsolution of 
\begin{equation}
-\Lu[y,u]\leq 0 \qquad \text{in } \R.  
\end{equation}
If $u$ attains a global maximum at $y_0\in\R$, then $u$ is constant on $\R$.
\end{theorem}
\begin{proof}
The proof can be found  in \cite[Thm. 2]{C}  if Assumption \ref{A2} (ii) holds and in in \cite[Thm. 4]{C}  if Assumption \ref{A2} (i) holds.
\end{proof}
\begin{remark} \label{subordinator}\upshape  
If Assumption \ref{A2} does not hold we cannot expect   the Strong Maximum Principle to be true. 
We show that this is always  the case when the L\'evy process $Z$ is a subordinator (see Remark \ref{sub1}).  
If $Z$ is a subordinator, by  the properties of the measure $\nu$ \eqref{sub}, the operator $\Lu$ in \eqref{l} can be equivalently written as
\begin{equation}\notag
\Lu[y,f]=-f'(y)\cdot\Big(y+\int_0^1 z \nu(dz)\Big)+\int_0^{+\infty} (f(y+z)-f(y))\nu(dz)=0 \quad \text{in }\R. \\
\end{equation} 
We set $c:=\int_0^1 z \nu(dz)$ and we take $f\in\mathcal{C}^2(\R)$, bounded and such that
\begin{equation}\notag
f'(y)>0 \quad \text{for every }y<-c \quad \text{and} \quad f(y)\equiv f(-c) \quad \text{for every }y\ge -c.\\
\end{equation}
We claim such function is a (classical) subsolution of $-\Lu[y,f]=0$. 
Indeed,
$$\int_0^{+\infty}(f(y+z)-f(y)) \nu(dz)\ge0$$
 for every $y$, since $f$ is nondecreasing, whereas the term $f'\cdot(y+c)$ is negative for $y<-c$ and $0$ for $y\ge-c$. On the other hand the maximum of $f$ at $-c$ propagates to the right but not to the left. Then the Strong Maximum Principle does not hold.
\end{remark}

\begin{theorem}[Liouville Property] 
\label{LP} Assume that the L\'evy measure $\nu$ satisfies Assumptions \ref{A1}, \ref{A3}, \ref{A2} and
consider the problem 
\begin{equation}
\label{-L}
- \Lu[y,V]=0, \qquad y\in\R, \\
\end{equation}
with $\Lu$ defined as in \eqref{l}. Then the following hold: 
\begin{enumerate}
\item\label{Sub} every bounded viscosity subsolution to (\ref{-L}) is constant;
\item\label{Sup} every bounded viscosity supersolution to (\ref{-L}) is constant.
\end{enumerate}
\end{theorem}

\begin{proof}
Let $V$ be a bounded subsolution to (\ref{-L}). We can assume w.l.o.g. that $V\ge0$. Let $\varphi$ 
be the Lyapunov function of  Lemma \ref{Lyapunov} and fix  $R>0$ such that $\varphi(y)>0$ for $|y|>R$. Define, for every $\eta>0$, 
\begin{equation*}
V_\eta(y)=V(y)-\eta\varphi(y)-\max_{|y|\leq  R}V. \\
\end{equation*} 
Observe that $V_\eta(y)\to -\infty$ as $|y|\to +\infty$ and moreover $V_\eta$ is USC. Then there exists $\overline{y}$ with $|\overline{y}|\geq R$ such that \[V_\eta(\overline{y})
\geq V_\eta(y) \qquad \forall |y|\geq R.\] Assume  that $|\overline{y}|>R$. Since $\varphi\in C^2$ and it is a strict supersolution to \eqref{-L} in $|y|>R$, we  get a contradiction to the fact that $V$ is a subsolution to \eqref{-L}. Then  $|\overline{y}|=R$. This implies that for all $\eta>0$ 
\[V(y)\leq \eta \varphi(y)+\max_{|y|\leq  R}V, \qquad\forall |y|\geq R\] and then, letting $\eta\to 0$, $V$ attains a maximum in the ball $|y|\leq R$. By the Strong Maximum Principle, theorem \ref{SMPFL}
$V$ is constant. 

The proof of (\ref{Sup}) for bounded supersolutions is analogous. 
\end{proof}

\section{The Hamilton-Jacobi-Bellman equation}
\label{4} 
The  HJB equation associated via dynamic programming to the value function \eqref{PO} of the control problem is 
\begin{equation}
\label{HJB}
- V^\eps_t+H(x,y,D_x V^\eps,D^2_{xx}V^\eps) -
\frac{1}{\eps}\Lu[y, V^\eps]+cV^\eps= 0 ,\quad \text{in} \quad(0,T) \times \R^n_+ \times \R ,
\end{equation}
where
\begin{equation}
\label{H}
H(x,y,p,X):=\min_{u\in U}\left\{- \frac 12 \tr ( \sigma(x,y,u)\sigma^T(x,y,u)X)-  f(x,y,u)\cdot p\right\}
\end{equation} 
and $\Lu$ is defined in \eqref{l}. It is a partial integrodifferential equation, briefly,  PIDE. The terminal condition associated to it is
\begin{equation}
\label{TC}
V^\eps(T,x,y) = g(x).
\end{equation}
Moreover, there is no natural boundary condition on the space boundary of the domain, \emph{i.e.}
$(0,T)\times\partial\R^n_+\times\R$.

The terminal  boundary value problem is well posed without prescribing any boundary condition because the value function is a solution in the set $(0,T)\times\R^n_+\times\R$. The irrelevance of the space boundary $(0,T)\times\partial\R^n_+\times\R$ is essentially due to the fact that $\R^n_+\times\R$ is an invariant set for the system \eqref{sde} for all admissible control functions (almost surely); that is, the state variable cannot exit this closed domain.  

\begin{proposition}\label{propvalue}
 For every $\eps>0$ the value function $V^\eps$ defined in \eqref{PO} is a continuous viscosity solution to \eqref{HJB} in the set $(0, T)\times\overline{\R^n_+}\times\R$
with terminal condition  \eqref{TC}. 
Moreover there exists a constant $C_T$ independent of $\eps$ such that
\begin{equation}
\label{growth}
 |V^\eps(t,x,y)|\leq C_T(1+|x|^2)\qquad \forall \ x\in\R^n, y\in\R, t\in [0,T].
 \end{equation}
  \end{proposition}
\begin{proof}
Using  the boundedness of $f,\sigma$ with respect to $y$ and   classical estimates on the moments of solutions of \eqref{sde} it is easy to show (see \cite[Lemma 3.1]{Ph2}) that 
\[\E|X(T)|^2\leq C'_T(1+|x|^2),\] where $(X(s),Y_\eps(s)) $ is the solution to \eqref{sde} with initial data $X(t)=x$, $Y_\eps(t)=y$. Then condition \eqref{GrowthC} implies 
 \eqref{growth}.

Again,  using moment estimates on the solutions to \eqref{sde} and the standing assumptions on the coefficients, it is possible to prove that $V^\eps$ is continuous for every $\eps$ (see \cite[Prop. 3.3]{Ph2}). Moreover, $V^\eps$ satisfies a dynamic programming principle (see \cite[Prop. 3.1 and Prop. 3.2]{Ph2}) and then by standard arguments  it is a viscosity solution to \eqref{HJB}. 
 
 Finally, 
 by condition \eqref{zero}, all the points of the boundary of $\R^n_+\times\R$ 
are irrelevant, according to a Fichera-type classification of boundary points for elliptic problems.  In other words, a sub- or supersolution in $(0, T)\times{\R^n_+}\times\R$ is automatically sub- or supersolution in $(0, T)\times\overline{\R^n_+}\times\R$. The argument is  detailed in \cite[Prop. 3.1]{BCM} in the case in which  the operator $\Lu$ is a local operator and 
it adapts without modifications to our case.   
\end{proof}

\section{Cell problem}
\label{5} 
In this Section we define the candidate limit Cauchy problem of the singular perturbed problem \eqref{HJB} as $\eps\to0$.  In particular we provide a formula for   the limit  Hamiltonian $\overline{H}$  as
\begin{equation}
\label{Heff}
\overline{H}(x,p,X)=\int_{\R} H(x,y,p,X)\mu(dy) 
\end{equation}
where $H$ is defined in \eqref{H} and $\mu$ is the invariant measure of the process defined in \eqref{Fast} (see Proposition \ref{lem:2}). 
 The main tool is the ergodicity of the process $Y_\lambda(t)$ proved in Section \ref{erg}. In the following  we will perform such construction in Theorem \ref{CellSold} using mainly PIDE methods. 

 In principle, for each fixed $(\bar{x},\bar{p},\bar{X})$ one expects the effective Hamiltonian $\overline{H}(\bar{x},\bar{p},\bar{X})$ to be a constant $k\in\R$ such that the \emph{cell problem}
\begin{equation}
\label{CP}
-\Lu[y,\chi]+H(\bar{x},y,\bar{p},\bar{X})=k \quad \text{in} \quad \R, \\
\end{equation}
has a viscosity solution $\chi$, called corrector. Actually, for our approach, it is sufficient to consider an approximate cell problem
\begin{equation}
\label{CPapproxd}
\delta\chi_\delta(y)-\Lu[y,\chi_\delta]+H(\bar{x},y,\bar{p},\bar{X})=0 \quad \text{in } \R, \\
\end{equation}
whose solution $\chi_\delta$ is also called approximate corrector.
\begin{theorem}
\label{CellSold}
For any fixed $(\bar{x},\bar{p},\bar{X})$ and $\delta>0$ the unique bounded  continuous viscosity solution $\chi_\delta(y)=\chi_{\delta;\bar{x},\bar{p},\bar{X}}(y)$ to \eqref{CPapproxd} is
\begin{equation}
\chi_\delta(y)=-\E\int_0^{+\infty} H(\bar{x},Y(t),\bar{p},\bar{X})e^{-\delta t}dt, \\
\end{equation} 
where $Y(t)$ solves the fast subsystem in \eqref{Fast} with $\lambda=1$ an $Y(0)=y$. Moreover,
\begin{equation} 
\lim_{\delta\to 0}\delta\chi_\delta(y)=-\int_{\R} H(\bar{x},y,\bar{p},\bar{X})\mu(dy)=: -\overline{H}(x,p,X), \\ 
\end{equation}  
locally uniformly in $y$, where $H$ is defined in \eqref{H} and $\mu$ is the unique invariant distribution of the process $Y$ defined in Proposition \ref{lem:2}.
\end{theorem}
\begin{proof}
For any fixed $(\bar{x},\bar{p},\bar{X})$, we define a function $h:\R\to\R$ as
\begin{equation} \notag
h(y):=H(\bar{x},y,\bar{p},\bar{X}) \quad \text{for every } y\in\R, \\
\end{equation}
which is bounded and Lipschitz by the standing assumptions. The existence and uniqueness of a bounded viscosity solution $\chi_\delta(y)$ follow from the Perron-Ishii method for PIDEs and the comparison principle in \cite{Sa}. Moreover by dynamic programming principle and standard arguments in Markov processes (see \cite{FS}, Chap. II.3)
we get that \begin{equation} 
\chi_\delta(y)=-\E\int_0^{+\infty} h(Y(t))e^{-\delta t}dt
\end{equation}
where $Y(t)$ solves the fast subsystem in \eqref{Fast} with $\lambda=1$. 
The second claim follows from Corollary \ref{Ab}.
\end{proof}

\begin{remark}\upshape
If there is no control $u$ in the system the Hamiltonian $H$ is a linear function of $(p,X)$  and $\overline H$ is obtained simply by averaging the coefficients
\[
\overline H (x,p,X)=- \frac 12  \tr \left( \int_{\R} \sigma(x,y)\sigma^T(x,y)\mu(dy) \,X\right) -  \int_{\R} f(x,y)\mu(dy) \cdot p .
\]
We will us this observation in the application to asset pricing, Section \ref{uncontrol}.
\end{remark}

\section{The Convergence Theorem}
\label{6}
We state now our main result, namely, the convergence theorem for the singular perturbation problem. We will prove that the value function $V^\eps(t,x,y)$, solution to \eqref{HJB}, converges locally uniformly, as $\eps\to 0$, to a function $V(t,x)$ which can be characterized as the unique solution of the limit problem
\begin{equation}
\label{HJBlim}
\left\{ \begin{array}{ll}
-V_t+\overline{H}\left(x,D_x V,D^2_{xx}V\right)+cV
=0 & \text{in} \,\, (0,T)\times\overline{\R^n_+}, \\
\\  V(T,x)=g(x) & \text{in} \,\, \overline{\R^n_+}. \\
\end{array}\right.
\end{equation}
where the Hamiltonian $\overline{H}$ has been defined respectively in \eqref{Heff}.

\begin{theorem}[Convergence Theorem] 
\label{ConvTheo}
The  value functions $V^\eps$ defined in \eqref{PO}  converge  as $\eps\to 0$ uniformly on compact subsets of $[0,T]\times\overline{\R^n_+}\times\R$ to the unique continuous viscosity solution to the limit problem \eqref{HJBlim} satisfying a quadratic growth condition in $x$, \emph{i.e.},
\begin{equation}\notag
\label{GrowthC'}
\exists\;K>0\,\,\text{such that for every } (t,x)\in[0,T]\times\R^n_+,\,\,|V(t,x)|\le K(1+|x|^2). \\
\end{equation}
\end{theorem}

\begin{proof} 
The proof is divided into several steps.
\begin{step}[Relaxed semilimits] \upshape
Recall that the functions $V^\eps$ are locally equibounded in $[0,T]\times\overline{\R^n_+}\times\R$, uniformly in $\eps$ (see Proposition \ref{propvalue}). We define the half-relaxed semilimits in $[0,T]\times\overline{\R^n_+}\times\R$ (see \cite{BC-D}, Chap. V):
\begin{align*}
& \underline{V}(t,x,y):= \liminf_{\eps\to0,t'\to t,x'\to x,y'\to y} V^\eps(t',x',y'), \\
& \overline{V}(t,x,y):= \limsup_{\eps\to0,t'\to t,x'\to x,y'\to y} V^\eps(t',x',y'),  
\end{align*}
for $t\le T$, $x\in\R^n_+$ and $y\in\R$. 
It is immediate to get by definitions that also $\underline{V}$ and $\overline{V}$ satisfy a quadratic growth condition in $[0,T]\times\overline{\R^n_+}\times\R$, \emph{i.e.} there exists a positive constant $K_1$ such that
\begin{equation} \notag
\label{GrowthC''}
|\underline{V}(t,x,y)|\le K_1(1+|x|^2) \quad \text{and } |\overline{V}(t,x,y)|\le K_1(1+|x|^2) \\ 
\end{equation}
for every $(t,x,y)\in[0,T]\times\overline{\R^n_+}\times\R$. \\
\end{step}
\begin{step}[$\underline{V},\overline{V}$ do not depend on $y$] \upshape
 We prove the claim only for $\overline{V}$, since the other case is completely analogous.

 First of all observe  that the function $\overline{V}(t,x,y)$ is a viscosity subsolution to 
\begin{equation}
\label{Lu}
-\Lu[y,V]=0 \quad \text{in}\,\R. \\
\end{equation} The detailed argument is in \cite[Thm. 5.1]{BCM}. Then arguing as \cite[Lemma II.5.17]{BC-D}, we get that for every fixed   $(\bar{t},\bar{x})$,  also the function $y\to\overline{V}(\bar{t},\bar{x},y)$ is a subsolution to \eqref{Lu}. 
So we can conclude by   the Liouville property Theorem \ref{LP}, since  $\overline{V}$ is  bounded in $y$,  that the function  $y\to\overline{V}(\bar{t},\bar{x},y)$ is constant  for every $(\bar{t},\bar{x})\in(0,T)\times\R^n_+$. Finally, using the definition it is immediate to see that this implies that also $\overline{V}(T,x,y)$ do not depend on $y$.

\end{step}
\begin{step}[$\underline{V}$ and $\overline{V}$ are super- and subsolutions of the limit PDE] \upshape
First we show 
 that $\underline{V}$ and $\overline{V}$ are super- and subsolution to  \eqref{HJBlim} in $(0,T)\times\R^n_+$. We prove it 
  only for $\overline{V}$ since the other case is completely analogous. The proof adapts the perturbed test function method introduced by Evans \cite{E} for  periodic homogenization and 
  developed in \cite{AB2, BCM} for singular perturbations. 
   We fix $(\bar{t},\bar{x})\in(0,T)\times\R^n_+$ and we show that $\overline{V}$ is a viscosity subsolution at $(\bar{t},\bar{x})$ of the limit PIDE. This means that if $\psi$ is a smooth function such that $\psi(\bar{t},\bar{x})=\overline{V}(\bar{t},\bar{x})$ and $\overline{V}-\psi$ has a maximum at $(\bar{t},\bar{x})$ then
\begin{equation}
-\psi_t(\bar{t},\bar{x})+\overline{H}(\bar{x},D_x\psi(\bar{t},D^2_{xx}\psi(\bar{t},\bar{x}))+c\overline{V}(\bar{t},\bar{x})\le 0.  
\end{equation}
Without loss of generality we assume that the maximum is strict in $B((\bar{t},\bar{x}),r)$ 
 and that $0<\bar{t}-r<\bar{t}+r<T$ and $\bar{x}_i>r$ for all $i$. 
We consider now the  Lyapunov  function $\phi\in\mathcal{C}^2(\R)$ as in Lemma \ref{Lyapunov}. 
By adding a constant to $\phi$ if necessary, we can assume that $-\Lu[y,\phi]\ge0$ in $\R$. We fix $\bar{y}\in\R$, such that $\phi(\bar{y})=\min_{\R} \phi$. 
We can also assume that $\phi(\bar{y})<\min_{|y-\bar{y}|\ge 
R}\phi$ for $R
$ sufficiently large.  
Let now $\eta>0$ and take $\delta>0$ sufficiently small such that if $\chi_\delta$ is the solution of \eqref{CPapproxd} at $(\bar{x},D_x\psi(\bar{t},\bar{x}),D^2_{xx}\psi(\bar{t},\bar{x}))$ (see Theorem \ref{CellSold}), then 
\begin{equation}
\label{deltabound}
|\delta\chi_\delta(y)+\overline{H}(\bar{x},D_x\psi(\bar{t},\bar{x}),D^2_{xx}\psi(\bar{t},\bar{x}))|\le\eta \quad \text{for any } y\in B(\bar{y}, R
). \\
\end{equation} 
We define the perturbed test function as
\begin{equation}
\psi^\eps(t,x,y):=\psi(t,x)+\eps\chi_\delta(y)+\phi(y). \\
\end{equation}
Observe that
\begin{equation}
\limsup_{\eps\to0,t'\to t,x'\to x, y'\to y} V^\eps(t',x',y')-\psi^\eps(t',x',y')=\overline{V}(t,x)-\psi(t,x)-\phi(y). \\
\end{equation}
Arguing as in \cite[Lemma V.1.6]{BC-D} we get 
sequences $\eps_n\to0$ and $(t_n,x_n,y_n)\in B:=B((\bar{t},\bar{x}),
r)
\times B(\bar y, R)$ such that $(t_n,x_n,y_n)\to(\bar{t},\bar{x},\hat{y})$ for some $\hat{y}\in B(\bar{y},R)$ such that $\phi(\hat{y})=\phi(\bar{y})$, and, as $n\to+\infty$,
\begin{equation} \notag
V^{\eps_n}(t_n,x_n,y_n)-\psi^{\eps_n}(t_n,x_n,y_n)\to\overline{V}(\bar{t},\bar{x})-\psi(\bar{t},\bar{x})-\phi(\hat{y}) \\
\end{equation}
and $(t_n,x_n,y_n)$ is a maximum of $V^{\eps_n}-\psi^{\eps_n}$ in $B$. 

Then, using the fact that $V^\eps$ is a subsolution to \eqref{HJB}, we get
\begin{equation}
-\psi_t+H(x_n,y_n,D_x\psi,D^2_{xx}\psi)+cV^\eps-\Lu[y_n,\chi_\delta]-\frac{1}{\eps_n}\Lu[y_n,\phi] \le0 \\
\end{equation}
where $V^\eps$, $\psi$, $\chi_\delta$ and $\phi$ (and their derivatives) are computed respectively in $(t_n,x_n,y_n)$, $(t_n,x_n)$ and in $y_n$. Using the fact that $\phi$ satisfies $-\Lu[\cdot,\phi]\ge0$, we get from the previous inequality that
\begin{equation}
-\psi_t+H(x_n,y_n,D_x\psi,D^2_{xx}\psi)+cV^{\eps_n}-\Lu[y_n,\chi_\delta] \le0. \\
\end{equation}
We now recall that $\chi_\delta$ solves the $\delta$-cell problem \eqref{CPapproxd}
, thus
\begin{gather}
-\psi_t(t_n,x_n)+H(x_n,y_n,D_x\psi(t_n,x_n),D^2_{xx}\psi(t_n,x_n)) \notag \\
-H(\bar{x},y_n,D_x\psi(\bar{t},\bar{x}),D^2_{xx}\psi(\bar{t},\bar{x}))-\delta\chi_\delta(y_n)+cV^{\eps_n}(t_n,x_n,y_n)\le0. \notag \\
\end{gather}
By taking the limit as $n\to+\infty$ the second and the third term of the left-hand side of this inequality cancel out. Next we use \eqref{deltabound} to replace $-\delta\chi_\delta$ with $\overline{H}-\eta$ and get that
\begin{equation}
-\psi_t(\bar{t},\bar{x})+\overline{H}(\bar{x},D_x\psi(\bar{t},\bar{x}),D^2_{xx}\psi(\bar{t},\bar{x}))+c\overline{V}(\bar{t},\bar{x})\le\eta. \\
\end{equation}
Finally, since $\eta>0$ is arbitrary, we conclude. \\
To prove that $\underline{V}$ is a supersolution to \eqref{HJBlim} we proceed exactly in the same way, just taking as a perturbed test function
\begin{equation} \notag
\psi^\eps(t,x,y)=\psi(t,x)+\eps\chi_\delta(y)-\phi(y). \\
\end{equation}

Finally, we claim that $\overline{V}$ and $\underline{V}$ are respectively a sub and a supersolution to \eqref{HJBlim} 
also at the boundary of $\R^n_+$. In this case it is sufficient to repeat exactly the same
argument as in \cite[Prop. 3.1
]{BCM} to get the conclusion, recalling that
the Hamiltonian $\overline{H}$ is defined as
\[
\overline{H}(x,p,X)=\int_{\R} \min_{u\in U}\left\{- \frac 12 \tr ( \sigma(x,y,u)\sigma^T(x,y,u)X)-  f(x,y,u)\cdot p\right\} \mu(dy)\]
and $f, \sigma$ satisfy \eqref{zero}.
\end{step}

\begin{step}[Uniform convergence] \upshape
We observe that by definition $\overline{V}\ge\underline{V}$ and that both $\overline{V}$ and $\underline{V}$ satisfy the same quadratic growth condition \eqref{GrowthC''}. Moreover, the Hamiltonian $\overline{H}$ defined in \eqref{Heff} inherits all the regularity properties of $H$ defined in the first section. Therefore, we can apply the  comparison result  in \cite{DL-L}  between sub- and supersolutions to parabolic PDE problems satisfying a polynomial growth condition, to deduce that $\overline{V}\le\underline{V}$. Therefore, 
\begin{equation} \notag
\overline{V}=\underline{V}=:V. \\
\end{equation}
In particular, $V$ is continuous, and by Lemma V.1.9 in \cite{BC-D}, this implies that $V^\eps$ converges locally uniformly to $V$.
\end{step}
 \end{proof} 
\begin{remark}\upshape
The idea of adding or subtracting  the Lyapunov function $\phi$ to the perturbed test function $\psi^\epsilon$ seems to be new and is an appropriate tool for extending the methods of \cite{E, AB2} from the periodic case to the present case of unbounded fast variables $y$. It applies also to the proof of Thm. 5.1 of \cite{BCM}, therefore filling in a gap of that proof.
\end{remark}
\begin{remark}
\label{effcontrol}\upshape
The solution $V$  of the limit Cauchy problem \eqref{HJBlim} can be represented as the value function of a new control problem obtained by a relaxation procedure proposed in \cite{BTer} for deterministic systems. Define the extended control set
\[
U^{ex}:= L^1((\R,\mu); U)
\]
and note that it contains a copy of $U$, given by the constant functions.
Extend to $U^{ex}$ the drift and the diffusion of the system \eqref{sde} as follows 
\[
\hat\sigma \hat\sigma^T(x,\beta):=\int_{\R} \sigma \sigma^T(x,y,\beta(y)) \mu(dy) , \quad \hat f (x,\beta):= \int_{\R} f((x,y,\beta(y)) \mu(dy), \quad \beta\in U^{ex}.
\]
Then the measurable selection argument in \cite{BTer} allows to prove that
\[
\int_{\R} H(x,y,p,X) \mu(dy) = \inf_{\beta\in U^{ex}}\left\{- \frac 12 \tr (\hat\sigma \hat\sigma^T(x,\beta) X)- \hat f (x,\beta)\cdot p\right\} ,
\] 
i.e., $\overline H$ is a Bellman Hamiltonian. By uniqueness of viscosity solutions to \eqref{HJBlim} we can conclude that
\[
V(t,x)=\sup_{\beta\in\mathcal{U}^{ex}}\E[e^{c(t-T)}g(\hat X(T))\,|\, \hat X(t)=x
] ,
\]
where $\hat X$ solves
\[
d\hat X(s)=\hat f(\hat X(s),\beta(s)) ds + \hat\sigma(\hat X(s),\beta(s))dW(s)
\]
and $\mathcal{U}^{ex}$ denotes the progressively measurable processes taking values in ${U}^{ex}$. This is an \emph{effective control problem} associated to the general multiscale control problem of Section \ref{2}. In the application to Merton's portfolio optimisation, 
Section \ref{effMert}, however, we will find a simpler representation of the limit control problem by exploiting the explicit form of the Hamiltonian in that case.
\end{remark}

\section{Applications and examples}%
\label{7}
\subsection{Asset pricing} \label{uncontrol}
We consider $n
$ underlying risky assets with price $X^i$ evolving according to the standard lognormal model:
\begin{equation}
\label{SDE1}
\left\{ \begin{array}{ll}
dX^i(t)=\alpha^i X^i(t)dt+\sqrt{2}X^i(t)\sigma_i(Y_\eps(t))dW^i(t), \quad i=1,\dots,n, \\
dY_\eps(t)=-\frac{1}{\eps}Y_\eps(t)dt+dZ(\frac{t}{\eps})
\end{array} \right.
\end{equation}
where $X^i(t_0)=x^i\ge0$, $\sigma_i:\R\to\R$ are bounded Lipschitz continuous function, 
$\sigma_i(y)\geq 0$ for all $y$, $i=1,\dots,n$, the processes $W=(W^1,\dots,W^n)$ and $Z$ are independent and, respectively, a standard $n$
-dimensional Brownian motion and a pure jumps L\'evy process with L\'evy measure satisfying the Assumptions \ref{A1}, \ref{A3}, \ref{A2}. 

The problem we consider here is the pricing of an European option given by a non-negative payoff function $g$ depending on the underlying $X^i$ and by a maturity time $T$. According to risk-neutral theory, to define a no-arbitrage derivative price we have to use an equivalent martingale measure $\Q$ under which the discounted stock prices $e^{-rt}X^i(t)$ are martingales, where $r$ is the instantaneous interest rate for lending or borrowing money. Nevertheless, we assume as in \cite{FPS} that the process $Z$ remains unchanged under the equivalent martingale measure $\Q$. Thus the system, under a risk-neutral probability $\Q$, \eqref{SDE1} writes as 
\begin{equation}
\label{SDE2}
\left\{ \begin{array}{ll} 
dX^i(t)=r X^i(t)dt +\sqrt{2}  X^i(t)\sigma_i(Y_\eps(t))
 dW^{i,\Q}(t),  \quad i=1,\dots,n, \\
dY_\eps(t)=-\frac{1}{\eps}Y_\eps(t)dt+dZ(\frac{t}{\eps}). \\
\end{array} \right.
\end{equation}
In this setting, an European contract has no-arbitrage price given by the formula
\begin{equation}
\label{Veps}
V^\eps(t,x,y):=\E^\Q[e^{c(t-T)}g(X(T))\,|\,X^i(t)=x^i,\,Y_\eps(t)=y], \quad 0\le t \le T
\end{equation}
where $c>0$ and the payoff function $g$ satisfies \eqref{GrowthC}. 

The (linear) HJB equation associated with the price function is
\begin{equation}
\label{HJBp}
\left\{ \begin{array}{ll}
-V^\eps_t 
-\sum_{i=1}^{n}x_i^2 \sigma_i^2(y) V^\eps_{x_ix_i} 
-r  x \cdot D_x V^\eps -\frac{1}{\eps}\Lu[y, V^\eps]+cV^\eps=0 \quad &\text{in } (0,T)\times\overline{\R^n_+}\times\R , \\ \\
V^\eps(T,x,y)=g(x) \quad &\text{in } \overline{\R^n_+}\times\R , \\
\end{array} \right.
\end{equation}
where
$\Lu$ is as in \eqref{l}. 
 
Since all the assumptions are satisfied, the convergence theorem holds, and the prices $V^\eps(t,x,y)$ converge locally uniformly, as $\eps\to0$, to the unique viscosity solution $V(t,x)$ of the limit equation
\begin{equation}
\label{HJBeff}
\left\{ \begin{array}{ll}
- V_t- \sum_{i=1}^{n}x_i^2 \int_{\R} \sigma_i^2(y)\mu(dy) \,V^\eps_{x_ix_i}
-r x \cdot D_xV+cV=0 \quad &\text{in } (0,T)\times\overline{\R^N_+}, \\ \\
V(T,x)=g(x), \quad &\text{in } \overline{\R^N_+}. \\
\end{array} \right.
\end{equation}
Then $V$ can be represented as 
\begin{equation}
V(t,x):=\E^\Q[e^{c(t-T)}g(\overline{X}(T))\,|\,\overline{X}(t)=x], \quad 0 \le t \le T, \\
\end{equation}
where $\overline{X}(t)$ satisfies the averaged \emph{effective system}
\begin{equation}
\label{effSDE}
d\overline{X}^i(t)=r\overline{X}^i(t)dt+\sqrt{2}  \overline{\sigma}_i\overline{X}^i(t)   dW^{i,\Q}(t), \\
\end{equation}
whose (constant) volatility $\overline{\sigma}_i$ is the so-called mean historical volatility for the $i$-th asset
\begin{equation}
\overline{\sigma}_i:=
\left(\int_{\R} \sigma_i^2(y) \mu(dy)\right)^{\frac 12} . 
\end{equation}
Therefore the limit of the pricing problem as $\eps\to0$ is a new pricing problem for the effective system \eqref{effSDE}.  
\begin{remark}
\label{correlation}
\upshape
The choice of a diagonal matrix $\sigma$ in \eqref{SDE1} is made only for notational simplicity. As in the previous sections and in \cite{BCM} we can replace the term $\sigma_i(Y_\eps(t))dW^i(t)$ in \eqref{SDE1} with the $i$-th component of $\sigma(Y_\eps(t))dW(t)$ for a $n\times r$ matrix $\sigma$, therefore allowing  that the components of the Brownian motion acting on different asset prices be correlated. In this case instead of an effective volatility $\overline{\sigma}_i$ for each asset we find an effective matrix $\overline{\sigma}$ such that $\overline{\sigma}\,\overline{\sigma}^T= \int_{\R} \sigma(y) \sigma^T(y)\mu(dy)$.
\end{remark}

\subsection{Merton's portfolio optimization problem} 
\subsubsection{The convergence result}
We consider now another classical problem in finance, the Merton's optimal portfolio allocation, under the assumption of fast oscillating stochastic volatility. 
We consider a financial market consisting of a nonrisky asset $S$ evolving according to the deterministic equation 
\begin{equation} \notag
dS(t)=r S(t) dt, \\
\end{equation}
with $r>0$, and $n$ risky assets $X^i(t)$ evolving according to the stochastic system
\eqref{SDE1}. We denote by $\W$ the wealth of an investor. The investment policy - which will be the control input - is defined by a progressively measurable process $u$ taking values in a compact set $U$, and $u^i_t$ represents the proportion of wealth invested in the asset $X^i(t)$ at time $t$. Then the wealth process evolves according to the following system:
\begin{equation}
\label{SDE3}
\left\{ \begin{array}{ll}
d\W(t)=\W(t) \bigl( r + \sum_{i=1}^n((\alpha^i-r) u^i(t))\bigr) dt + \sqrt{2}\W(t) \bigl(\sum_{i=1}^N u^i(t)\sigma_i(Y_\eps(t))\bigr) dW(t), \\
dY_\eps(t)=-\frac{1}{\eps}Y_\eps(t) dt+dZ(\frac{t}{\eps})
\end{array} \right.
\end{equation}
where $\W(t_0)=w>0$. \\

The Merton's problem consists in choosing a strategy $u_.$ which maximizes a given utility function $g$ at some final time $T$. In particular the problem can be described in terms of the value function 
\begin{equation}\notag
V^\eps(t,w,y):=\sup_{u_.\in\mathcal{U}}\E[g(\W(T))\,|\,\W(t)=w,Y_\eps(t)=y]. \\
\end{equation}
Tipically the utility functions in financial applications are chosen in the class of HARA (hyperbolic absolute risk aversion) functions $g(w)=a(bw+c)^\gamma$, where $a,b,c$ are constants, and $\gamma\in(0,1)$ is a given coefficient called the relative risk premium coefficient. The HJB equation associated with the Merton value function is
\begin{equation}\notag
-V^\eps_t+H_M(w,y,D_wV^\eps,D^2_{ww}V^\eps)-\frac{1}{\eps}\Lu[y,V^\eps]=0, \\
\end{equation}
in $(0,T)\times\R_+\times\R$, complemented with the terminal condition $V^\eps(T,w,y)=g(w)$. The integro-differential operator $(1/\eps)\Lu$ is the infinitesimal generator of the process $Y_\eps$, with $\Lu$ as in \eqref{l}, 
and $H_M(w,y,p,X)$ is defined as
\begin{equation}\notag
H_M(w,y,p,X)=\min_{u\in U}\left\{-\left(\sum_{i=1}^n u^i\sigma_i(y)\right)^2 w^2 X - \left[r+\sum_{i=1}^n (\alpha^i-r)u^i\right]w p\right\}. \\
\end{equation} 
Our main theorem applies also in this case and states that the value function $V^\eps$ converges locally uniformly to the unique solution of the limit problem 
\begin{equation}
\label{SDE4}
\left\{ \begin{array}{ll}
- V_t+\int_{\R}H_M(w,y,D_wV,D^2_{ww}V)\mu(dy)=0 \quad &\text{for } t\in(0,T), \quad w>0, \\
V(T,w)=g(w) \quad &\text{for } w>0, \\
\end{array} \right.
\end{equation}
where $\mu$ is the invariant distribution associated with the fast subsystem. This convergence result in the framework of L\'evy processes is new in the literature and extends and complements the results stated in \cite{BCM} for fast volatility processes driven by Brownian motion.   
\begin{remark}
\upshape
As in the Remark \ref{correlation} of the previous section we can allow the correlation among the noises acting on different assets, as in \cite{BCM}, Sect. 6.2.
\end{remark}

\subsubsection{The effective Merton's problem}
\label{effMert}
Next we want to interpret the effective PDE in \eqref{SDE4} as the HJB equation of a suitable \emph{effective control problem}, simpler than the general one described in Remark \ref{effcontrol}.
For simplicity we restrict ourselves to the case of a single risky asset, \emph{i.e.}, $n=1$. The equation for the wealth becomes
\begin{equation}\notag
d\W(t)=\W(t) (r+(\alpha-r)u(t)) dt + \sqrt{2}\W(t)u(t)\sigma(Y_\eps(t)) dW(t), \quad \alpha>r, \\
\end{equation} 
and the HJB equation for $V^\eps$ is
\begin{equation}\label{PIDE5}
-V^\eps_t-\max_{u\in U}\Bigl\{u^2 \sigma^2(y) w^2 V^\eps_{ww} + [r+(\alpha-r)u] w V^\eps_w\Bigr\}=\frac{1}{\eps}\Lu[y,V^\eps]. \\
\end{equation}
Therefore the effective PDE  in \eqref{SDE4} is
\begin{equation}\label{braces}
- V_t- r w V_w -\int_{\R}\max_ {u\in U}\Bigl\{u^2\sigma^2(y)w^2 V_{ww} + 
(\alpha-r)u
w V_w\Bigr\}\mu(dy)=0. \\
\end{equation} 
Now we assume that the utility function $g$ is increasing and concave with $g''<0$. Then one expects that the solution $V$ is also increasing in the initial wealth $w$ with $V_{ww}<0$. In this case the expression in braces in \eqref{PIDE5} and \eqref{braces} is a concave parabola in $u$ and the maximum in the Hamiltonian is easily computed if it is attained in the interior of $U$, e.g., if $U$ is very large. Then one gets, at least formally,
\[
H_M(w,y,p,X)=\frac{(\alpha - r)^2 p^2}{4\sigma^2(y) X}
\]
and the effective PDE \eqref{braces} becomes
\begin{equation}\notag
- V_t- r w V_w + \int_{R}\frac{1}{\sigma^2(y)}\mu(dy) \frac{(\alpha - r)^2 V_{w}^2}{4 V_{ww}}=0  .
\end{equation}
This is the HJB equation of  the \emph{Merton's problem with constant volatility $\overline{\sigma}>0$}, where the wealth dynamics is 
\begin{equation}\label{constant_v}
d\W(t)=\W(t)(r+(\alpha-r)u(t))dt+\sqrt{2}\W(t) u(t)\overline{\sigma} dW(t), \\
\end{equation}
if and only if
\begin{equation}\label{sigma}
\overline{\sigma}:=\Bigl(\int_{\R}\frac{1}{\sigma^2(y)}\mu(dy)\Bigr)^{-\frac{1}{2}}. \\
\end{equation}
Therefore this is  the correct parameter to use in a Merton model with constant volatility if we consider it as an approximation of a model with fast and ergodic stochastic volatility. We can call it the \emph{effective Merton problem}.
We point out that the \emph{effective volatility} $\overline{\sigma}$ for the Merton problem is the harmonically averaged long-run volatility, that is smaller than the usual mean historical volatility
\begin{equation}\notag
\tilde{\sigma}:=\Bigl(\int_{\R} \sigma^2(y) \mu(dy)\Bigr)^{\frac{1}{2}} 
\end{equation}
in uncontrolled systems, see Section \ref{uncontrol}. Therefore the use of the correct parameter $\overline{\sigma}$ in the model leads to an increase of the value function, \emph{i.e.}, of the optimal expected utility. 

\subsubsection{The solution for HARA utility}
In some cases the  Cauchy problem for the effective equation\eqref{SDE4} can be solved explicitly, giving also a more rigorous derivation of the formula \eqref{sigma} for the effective volatility.  Let us take as a constraint on the control $u(t)$ 
the interval 
\begin{equation} \notag
U:=[R_1,R], \quad \text{with } -R\le R_1\le 0 < R, \\
\end{equation}
and as terminal cost the HARA function
\begin{equation} \notag
\label{UtFunct}
g(w)=a\frac{w^\gamma}{\gamma}, \quad 0<\gamma<1, \quad a>0.
\end{equation}
Since the terminal condition is now 
$V(T,w)=a
{w^\gamma}/{\gamma},$ 
 we look for solutions of the form 
 $$V(t,w)=\frac{w^\gamma}{\gamma}v(t) \quad \text{with }\; v(t)\ge 0.
 $$
  By plugging it into the Cauchy problem we get
\begin{equation} \notag
\dot{v}=-\gamma\overline{h}v, \quad v(T)=a, \quad \overline{h}:=r+\int_{\R}\max_{u\in U}\bigl\{(\alpha-r)u+(\gamma-1)\sigma^2(y)u^2\bigr\}\mu(dy). \\
\end{equation}

Therefore the uniqueness of solution gives 
\begin{equation}
\label{Sol}
V(t,w)=a\exp{\Bigl\{\gamma\bar{h}(T-t)\Bigr\}}\frac{w^\gamma}{\gamma}, \quad 0<t<T. \\
\end{equation}
We compute the rate of exponential increase $\overline{h}$ and get 
\begin{align*}\notag
\overline{h}=r\,+&\int_{\{y\,:\,2R(1-\gamma)\sigma^2(y)<\alpha-r\}}\bigl[(\alpha-r)R+(\gamma-1)R^2\sigma^2(y)\bigr]\mu(dy) \notag \\
&+\int_{\{y\,:\,2R(1-\gamma)\sigma^2(y)\ge\alpha-r\}}\frac{(\alpha-r)^2}{4(1-\gamma)\sigma^2(y)}\mu(dy). \notag \\
\end{align*}
This formula simplifies considerably if the $\mu$-probability of the set $\{y\,:\,2R(1-\gamma)\sigma^2(y)\ge\alpha-r\}$ is 1, e.g., 
 for large upper bound $R$ on the control. In fact we get
\begin{equation}\notag
\overline{h}=r+  \frac{(\alpha-r)^2}{4(1-\gamma)}\int_{\R} \frac{1}{\sigma^2(y)}\mu(dy) . \\
\end{equation}
With this  expression for $\overline{h}$ the function $V$ given by \eqref{Sol}  coincides with the classical Merton formula solving the 
problem with constant volatility $\overline{\sigma}>0$ (for $2R(1-\gamma)\overline{\sigma}\ge\alpha-r$) if and only if $\overline{\sigma}$ is given by \eqref{sigma}, i.e.,

\begin{equation}\notag
V(t,w) = a\exp{\Bigl\{\gamma\left[r+\frac{(\alpha-r)^2}{4(1-\gamma)\overline{\sigma}^2}\right](T-t)\Bigr\}}\frac{w^\gamma}{\gamma}. \\
\end{equation}

\section*{Acknowledgements} We are grateful to Carlo Sgarra for pointing out to us the literature on non-Gaussian models of stochastic volatility and to Tiziano Vargiolu
for several discussions on L\'evy processes. The main results of this research were presented also in the third author's Master Thesis, defended in February 2014.

\end{document}